\documentclass[12pt]{article}
\usepackage{amsmath, amsthm, amssymb,amscd, verbatim,hyperref,mathrsfs,authblk,cleveref, tikz,graphicx}
\usepackage[all]{xy}
\theoremstyle{plain}
\newtheorem{thm}{Theorem}[section]

\newtheorem{lem}[thm]{Lemma}
\newtheorem{cor}[thm]{Corollary}

\newtheorem{prop}[thm]{Proposition}

\theoremstyle{definition}
\newtheorem{defn}[thm]{Definition}

\theoremstyle{remark}

\newcommand{\nc}{\newcommand} 
\nc{\hb}{\mathbb} 
\nc{\M}{\mathcal} 
\nc{\mf}{\mathfrak}
\nc{\mbf}{\mathbf}
\nc{\DMO}{\DeclareMathOperator}

\newbox\noforkbox \newdimen\forklinewidth
\setbox0\hbox{$\textstyle\smile$}
\setbox1\hbox to \wd0{\hfil\vrule width \forklinewidth depth-2pt
 height 10pt \hfil}
\setbox\noforkbox\hbox{\lower 2pt\box1\lower 2pt\box0\relax}
\def\anchor{\mathop{\copy\noforkbox}\limits}
 
\setbox0\hbox{$\textstyle\smile$}
\setbox1\hbox to \wd0{\hfil{\sl /\/}\hfil}
\setbox2\hbox to \wd0{\hfil\vrule height 10pt depth -2pt width
               \forklinewidth\hfil}
\newbox\doesforkbox
\setbox\doesforkbox\hbox{\box1 \lower 2pt\box2\lower2pt\box0\relax}
\def\nanchor{\mathop{\copy\doesforkbox}\limits}
 
\nc{\cA}{{\M A}} \nc{\cB}{{\M B}} \nc{\cC}{{\M C}} \nc{\cD}{{\M D}}
\nc{\cE}{{\M E}} \nc{\cF}{{\M F}} \nc{\cG}{{\M G}} \nc{\cH}{{\M H}}
\nc{\cI}{{\M I}} \nc{\cJ}{{\M J}} \nc{\cK}{{\M K}} \nc{\cL}{{\M L}}
\nc{\cM}{{\M M}} \nc{\cN}{{\M N}} \nc{\cO}{{\M O}} \nc{\cP}{{\M P}}
\nc{\cQ}{{\M Q}} \nc{\cR}{{\M R}} \nc{\cS}{{\M S}} \nc{\cT}{{\M T}}
\nc{\cU}{{\M U}} \nc{\cV}{{\M V}} \nc{\cW}{{\M W}} \nc{\cX}{{\M X}}
\nc{\cY}{{\M Y}} \nc{\cZ}{{\M Z}}
\nc{\Aa}{{\hb A}} \nc{\Cc}{{\hb C}} \nc{\Gg}{{\hb G}}
\nc{\Nn}{{\hb N}} \nc{\Pp}{{\hb P}} 
\nc{\Qq}{{\hb Q}} \nc{\Rr}{{\hb R}} \nc{\Zz}{{\hb Z}}
\nc{\mfa}{{\mf a}} \nc{\mfb}{{\mf b}} \nc{\mfk}{{\mf k}}
\nc{\mfm}{{\mf m}} \nc{\mfp}{{\mf p}} \nc{\mfq}{{\mf q}}
\nc{\mfr}{{\mf r}}
\nc{\fP}{{\mf P}}
\DMO*{\trdeg}{td}
\DMO*{\spec}{Spec}
\DMO*{\fork}{\nanchor}
\DMO*{\dnf}{\anchor}
\DMO{\RU}{RU}
\DMO{\deter}{det}
\DMO{\RM}{RM}
\DMO{\RC}{RC}
\DMO{\Real}{Re}
\DMO{\Imag}{Im}
\DMO{\tr}{tr}
\DMO{\qc}{QC}
\DMO{\Hu}{Hull}
\DMO{\leg}{length}
\DMO{\area}{area}
\DMO{\dia}{diameter}
\DMO{\iso}{Iso}
\DMO{\dis}{dist}
\DMO{\grad}{grad}
\DMO{\vol}{volume}
\DMO{\gra}{grad}
\DMO{\hd}{nbhd}
\DMO{\dv}{div}
\DMO{\Psl}{PSL}
\nc{\Mb}{\mathfrak^{2b/\delta}_{K_x}}
\nc{\Ma}{\mathfrak^{2a/\delta}_{K_x}}
\nc{\dif}{\mathrm{d}}
\nc{\G}{\Gamma}
\nc{\g}{\gamma}
\nc{\D}{\nabla}
\nc{\p}{\partial}
\nc{\DD}{\Delta^2}
\nc{\pp}{\partial^2} 
\nc{\de}{\delta}
\nc{\td}[2]{\trdeg{({#1}/{#2})}}
\nc{\dtd}[2]{\trdeg_{\delta}{({#1}/{#2})}}
\nc{\dspec}[1]{\spec_{\delta}{#1}}
\nc{\ddim}[1]{\dimen_{\delta}{#1}}
\nc{\gens}[1]{\langle {#1} \rangle}        
\nc{\gen}[2]{ {#1} \langle {#2} \rangle } 
\nc{\form}{\Omega}
\nc{\set}[1]{\left\{ {#1} \right\}}
\nc{\mr}{\hat}
\nc{\pr}{\partial}
\nc{\bc}[3]{\cB^{#1}({#2},{#3})=B^{#1}_{#2}(C^{#2}_{#3}(t)+B^{#1}_{#3}(C^{#2}_{#3}(t))} 
\nc{\tuple}[2]{{#1},\ldots,{#2}} \nc{\ptu}[2]{{#1}:\ldots:{#2}}

\nc{\maps}[3]{{#1}\!:\!{#2}\rightarrow{#3}}
\nc{\map}[2]{{#1}\rightarrow {#2}} \nc{\res}[2]{{#1} |_{#2}}
\nc{\imbed}{\hookrightarrow}
\title{The classification of Kleinian groups of Hausdorff dimensions at most one}
\author{Yong Hou\footnote{Supported by Ambrose Monell Fundation} }
\date{}
\affil{Institute for Advanced Study}
\affil{Princeton University}
\begin{document}
\maketitle
\begin{abstract}
In this paper we provide the complete classification of Kleinian group of Hausdorff dimensions less than $1.$ In particular, we prove that every purely loxodromic Kleinian groups of Hausdorff dimension $<1$ is a classical Schottky group. This upper bound is sharp. As an application,  the result of \cite{H} then implies that
every closed Riemann surface is uniformizable by a classical Schottky group. The proof relies on the result of Hou \cite{Hou}, and space of rectifiable $\G$-invariant closed curves.
\end{abstract}
\setcounter{tocdepth}{1}
\section{Introduction and Main Theorem}
We take Kleinian groups to be finitely generated, torsion-free, discrete subgroups of $\text{PSL}(2,\mathbb{C}).$ The main theorem is:
\begin{thm}[Classification]\label{main}
Any purely loxodromic Kleinian group $\G$ with limit set of Hausdorff dimension $<1$ is a classical Schottky group. This bound is sharp.
\end{thm}
We note that by Selberg lemma, torsion-free is not really a restriction, since any finitely generated discrete subgroup of $\text{PSL}(2,\mathbb{C})$ 
has a finite index torsion-free subgroup. \par
As an application, we have the following Corollary \ref{bers}, which is a resolution of a folklore problem of Bers on classical Schottky group uniformization of closed Riemann surface. Corollary \ref{bers} follows from the work of Hou \cite{H}:
\begin{thm}[Hou \cite{H}]\label{AH}
Every closed Riemann surface is uniformizable by a Schottky group of Hausdorff dimension $<1$, i.e. every point in moduli space has a Hausdorff dimension $<1$  fiber in the Schottky space. 
\end{thm}
\begin{cor}(Uniformization)\label{bers}
Every closed Riemann surface can be uniformized by a classical Schottky group.
\end{cor}
\subsection{Strategy of proof}
First let us recall the result of Hou \cite{Hou}:
\begin{thm}[Hou]\label{Hou}
There exists $\lambda>0$ such that any Kleinian group with limit set of Hausdorff dimension $<\lambda$ is a classical Schottky group
\end{thm}
Define $H_c=\sup\{\lambda|\text{satisfies Theorem \ref{Hou}}\}.$ $H_c$ is the maximal parameter such that if $\G$ is a Schottky group of Hausdorff dimension $<H_c$ then $\G$ is classical Schottky group. Hence Theorem \ref{main} can be rephrased as:
$H_c\ge 1.$

We prove by contradiction, so from now on and throughout the paper we assume that $H_c<1,$ then we will show that $H_c$ is not maximal.\par

Recall that the Hausdorff dimension function on the
Schottky space of rank $g$ is real analytic. 
It is a consequence of Theorem \ref{Hou} that, $\mathfrak{J}^{H_c}_{g}$ 
(rank-$g$ Schottky groups of Hausdorff dimension $<{H_c}$, see Section $2$) is a
$3g-3$ dimensional open and connected submanifold of $\mathfrak{J}_{g}$, the rank-$g$ Schottky space. \par
The proof is done as follows. First we note that, $\partial\mathfrak{J}_{g}^{H_c}$ must contain a non-classical Schottky group, otherwise $H_c$ is not maximal by definition, see Proposition \ref{H}. Second we show that, 
 if $H_c<1$ then every element of the boundary $\partial\mathfrak{J}_{g}^{H_c}$ is either a classical schottky group or, it is not a Schottky group, Lemma \ref{singular}. This contradicts the first fact, hence we must have $H_c\ge 1.$ The bulk of the paper is devoted to proof the second fact, which we now summarize the idea in the following.\par
It is a result of Bowen \cite{Bowen} that, a Schottky group $\G$ has Hausdorff dimension $<1,$ if and only if there exist a \emph{rectifiable} $\G$-invariant closed curve. Let $\mathscr{R}(S^1,W)$ be the space of bounded length closed curves which intersects the compact set $W\subset\mathbb{C}$ and equipped with Fr\'echet metric. It is complete space, see Section $2$. 
We show that if $H_c<1$ then, every quasi-circle with bounded length of $\G$ is the limit of a sequence of quasi-circles of $\G_n$ in $\mathscr{R}(S^1,W).$ We also show that 
if $\G$ is a Schottky group, then every quasi-circle of $\G$ has an open neighborhood in the relative topology of $\Psi_\G$ (see section 3) such that, every element of the open neighborhood is a quasi-circle  of $\G.$ We also define linearity and transversality invariant for quasi-circles, and show that quasi-circles of classical Schottky groups preserve these invariants, and non-classical Schottky groups do not have transverse linear quasi-circles.\par
Given a quasi-circle of a Schottky group $\G$, we show that there exists an open neighborhood (in the relative topology of space of rectifiable curves with respect to Frechet metric) about the quasi-circle such that, every point in the open neighborhood is a quasi-circle of $\G,$ see Lemma \ref{open}. Next assume that we have a sequence of
classical Schottky groups $\G_n\to\G$ to a Schottky group, and are all of Hasudorff dimensions less than one. 
We then study singularity formations of classical fundamental domains of $\G_n$ when $\G_n\to\G.$ These singularities are of three types: tangent, degenerate,
and collapsing. We show that all these singularities will imply that there exists a quasi-circle such that, every open neighborhood about this quasi-circle will contain some points which is \emph{not} a quasi-circle. Essentially, the existence of a singularity will be \emph{obstruction} to the existence of any open neighborhood that are of quasi-circles, 
see Lemma \ref{singular}. Hence it follows from these results that, if $\G_n\to\G$ with $\G_n$ classical and, all Hausdorff dimensions are of less than one then $\G$ must be a classical Schottky group.
\centerline{Acknowledgement}\par
This work is made possible by unwavering supports and insightful conversations from Peter Sarnak, whom I'm greatly indebted to. It is the groundbreaking works of Peter Sarnak that has guided the author to study this problem at first place. I wish to express my deepest gratitude and sincere appreciation to Dave Gabai for the continuous of amazing supports and encouragements which allowed me to complete this work.\par
I want to express my sincere appreciation to the referee for detailed reading and, helpful comments and suggestions. I also want express sincere appreciation to Ian Agol, Matthew de Courcy-Irland, for reading of the previous draft. \par
This paper is dedicated to my father: ShuYing Hou.
\section{Quasi-circles and generating Jordan curves }
Schottky group $\G$ of rank $g$  is defined as convex-cocompact discrete faithful representation of the free group $\mathbb{F}_g$ in $\text{PSL}(2,\mathbb{C}).$ It follows that 
$\G$ is freely generated by purely loxodromic elements $\{\g_i\}_1^g$. This implies we can find collection of 
open topological disks $ D_{i}, D_{i+g}, 1\le i\le g$ of disjoint closure $\bar D_i\cap\bar D_{i+g}=\emptyset$ in the Riemann sphere $\partial\mathbb{H}^3=\overline{\mathbb{C}}$ with
boundary curves $\partial\bar D_{i}=c_{i},\partial \bar D_{i+g}=c_{i+g}.$ By definition $c_{i},c_{i+g}$ are closed Jordan curves in Riemann sphere 
$\partial\mathbb{H}^3,$ such that $\g_i(c_{i})=c_{i+g}$ and $\g_i(D_{i})\cap D_{i+g}=\emptyset.$ Whenever there exists a
set $\{\g_1,...,\g_g\}$ of generators with all $\{c_{i},c_{i+g}\}_1^g$ as circles, then it is 
called a classical Schottky group with $\{\g_1,...,\g_g\}$ classical generators.\par
Schottky space $\mathfrak{J}_g$ is defined as space of all rank $g$ Schottky groups up to conjugacy by $\text{PSL}(2,\mathbb{C}).$ By normalization,
we can chart $\mathfrak{J}_g$ by $3g-3$ complex parameters. Hence $\mathfrak{J}_g$ is $3g-3$ dimensional complex manifold. The bihomolomorphic 
$\text{Auto}(\mathfrak{J}_g)$ group is $\text{Out}(\mathbb{F}_g),$ which is isomorphic to quotient of the handle-body group. Denote by $\mathfrak{J}_{g,o}$ the set of all elements of $\mathfrak{J}_g$ that are classical Schottky groups. Note that $\mathfrak{J}_{g,o}$ is open in $\mathfrak{J}_g.$ On the other hand it is nontrivial result due to Marden that $\mathfrak{J}_{g,o}$ is non-dense subset of $\mathfrak{J}_g.$ However, it follows from
Theorem \ref{Hou}, $\mathfrak{J}^\lambda_{g}\subset\mathfrak{J}_{g,o}$ is $3g-3$ dimensional open connected submanifold. Here  $\mathfrak{J}^\lambda_{g}$ denotes space of Schottky groups of Hausdorff dimension $<\lambda.$\par

\noindent{\bf  Some notations:} 
\begin{itemize}

\item
Given $\G$ a Kleinian group, we denote by $\Lambda_\G$ and $\Omega_\G$ and $\mathfrak{D}_\G$ its limit set, region of discontinuity,
and Hausdorff dimension respectively throughout this paper. 
\item
Given a fundamental domain $\mathcal{F}$ of $\G$, we denote the orbit of $\mathcal{F}$ under actions of $\G$ by $\mathcal{F}_\G.$ We also say $\mathcal{F}_\G$
is a classical fundamental domain of classical Schottky group if $\partial\mathcal{F}$ are disjoint circles.
\end{itemize}
\begin{defn}[Quasi-circles]
Given a geometrically finite Kleinian group $\G,$ a closed $\G$-invariant Jordan curve that contains the limit set  $\Lambda_\G$ is called \emph{quasi-circle} of $\G$.
\end{defn}
\noindent{\bf Remark:}
From now on, we make the \emph{global assumption throughout this paper} that $H_c<1$, and \emph{all Schottky groups} are of Hausdorff dimension $\mathfrak{D}_\G\le H_c,$ if not stated otherwise. \par
Next we give a construction of quasi-circles of $\G$ which is a generalization of the construction by Bowen \cite{Bowen}. \par
Let $\mathcal{F}$ be a fundamental domain of $\G$, and $\{c_i\}_1^{2g}$ be the collection of $2g$ disjoint Jordan curves comprising $\partial\mathcal{F}.$  Let $\zeta$ denote collection of arcs $\zeta=\{\zeta_i\}$ connecting points $p_i\in c_i,p'_i\in c_{i+g}$  for $1\le i\le g,$ and 
arcs on $c_{i+g}$ that connects $p_{i+g}$ to $\g_i(p_i)$ and $p'_{i+g}\in c_{i+g}$ to $\g_i(p'_i).$
So $\zeta$ is a set of $g$ disjoint curves connecting disjoint points on collection of Jordan curves of $\partial\mathcal{F}$ (Figure 1 ). \par
\begin{figure}[ht!]
\centering
\includegraphics[width=45mm]{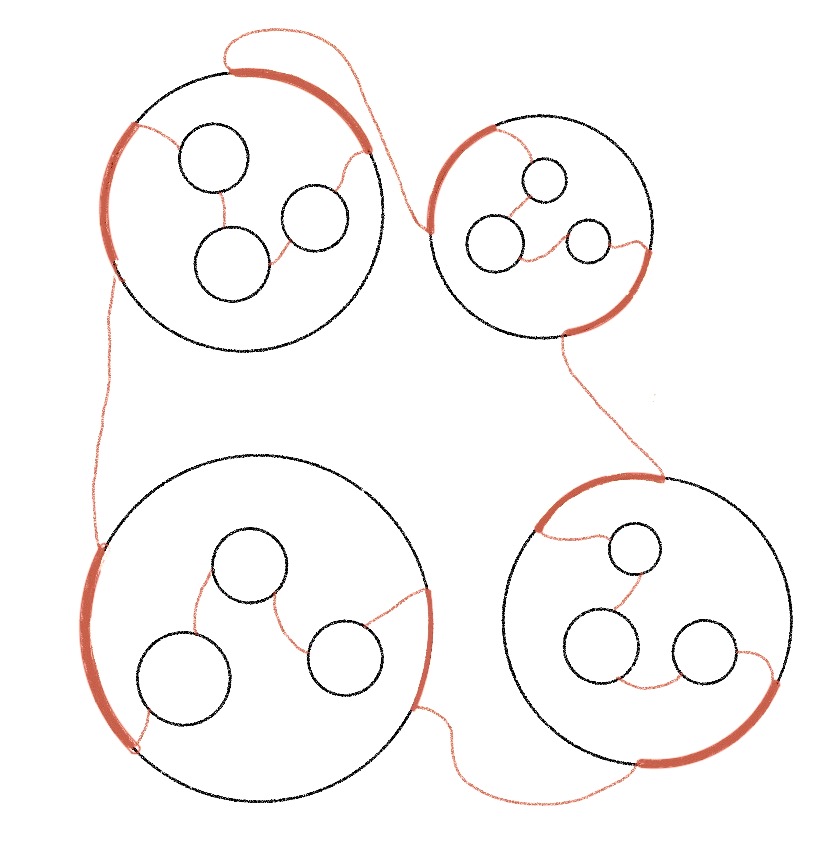}  
\caption{Quasi-circle \label{overflow}}
\end{figure}
$\eta_\G=\Lambda_\G\cup\cup_{\g\in\G}\g(\zeta)$ defines a $\G$-invariant closed curve containing $\Lambda_\G.$ 
$\eta_\G$ defines a quasi-circle of $\G$. Obviously there are infinitely many quasi-circles and different $\zeta$ gives a different quasi-circles. Note that, the simply connected regions $\mathbb{C}\setminus\eta_\G$ gives the Bers simultaneous uniformization of Riemann surface $\Omega_\G/\G.$

\begin{defn}[Generating curve]\label{marking}
Given a quasi-circle $\eta_\G$ of $\G.$ We say a collection of disjoint curves $\zeta$ is a generating curve of $\eta_\G$, if $\eta_\G$ can be generated by $\zeta.$
\end{defn}

Note that the quasi-circles constructed in \cite{Bowen}, which requires that $p_{i+g}$ is a imagine of $p_i$ under element of $\g_i$, is a subset of the collection that we have defined here. In fact, this generalization is also used for the construction of quasi-circles of non-classical Schottky groups. 

\begin{prop}\label{marking}
Every quasi-circle $\eta_\G$ of $\G$ is generated by some generating curves $\zeta.$
\end{prop}
\begin{proof}
Let $\eta_\G$ be a quasi-circle of $\G.$ Let $\mathcal{F}$ be a fundamental domain of $\G.$ Set $\xi=\mathcal{F}\cap\eta_\G$ and 
$\bar\xi=\partial\mathcal{F}\cap\eta_\G.$ Then $\xi,\bar\xi$ consists of collection of disjoint curves which only intersects along $\partial\mathcal{F}.$
Hence we have, $\g(\xi)\cap\g'(\xi)=\emptyset$ for $\g,\g'\in\G$ with $\g\not=\g'.$ Since $\overline{\G(\xi)}-\G(\xi)=\Lambda_\G$ we have have 
$\eta_\G=\Lambda_\G\cup\G(\xi\cup\bar\xi),$ hence $\xi\cup\bar\xi$ is a generating curve of $\eta_\G.$

\end{proof}

\begin{defn}[Linear quasi-circle]
We call a quasi-circle $\eta_\G$ \emph{linear} if, 
 $\eta_\G\backslash\Lambda_\G$ consists of points, circular arcs or  lines.
\end{defn}

Note that, if $\eta_\G$ is linear then there exists  $\mathcal{F}_\G$ such that, $\eta_\G\cap \mathcal{F}_\G$ and $\eta_\G\cap \partial\mathcal{F}_\G$ 
are piece-wise circular arcs or lines.\par
We say an arc $\zeta\subset\eta_\G\cap \mathcal{F}_\G$
is orthogonal if the tangents at intersections on $\partial\mathcal{F}_\G$
are orthogonal with $\partial\mathcal{F}_\G,$ and an arc $\xi\subset\eta_\G\cap \partial\mathcal{F}_\G$ is parallel if $\xi\subset\partial\mathcal{F}_\G.$

\begin{defn}[Right-angled quasi-circle]
Given a linear quasi-circle $\eta_\G$ of $\G$, if all linear arcs intersect at right-angle then we say $\eta_\G$ is \emph{right-angled quasi-circle}.
\end{defn}

\begin{defn}[Transverse quasi-circle]
Given a quasi-circle $\eta_\G$ of $\G$, we say $\eta_\G$ is \emph{transverse} quasi-circle if $\eta_\G$ intersects $\partial\mathcal{F}_\G$ orthogonally for some 
$\mathcal{F}_\G$ and, $\eta_\G$ have no parallel arc. Otherwise, we say $\eta_\G$ is non-transverse.

\end{defn}

\begin{defn}[Parallel quasi-circle]
Given a quasi-circle $\eta_\G$ of $\G$, we say $\eta_\G$ is \emph{parallel} quasi-circle if there exists some arc $\eta$ of $\eta_\G$ such that 
$\eta\subset\partial\mathcal{F}_\G$

\end{defn}

\begin{prop}
Transverse quasi-circles always exists for a given Schottky group $\G.$ 
\end{prop}
\begin{proof}
Let $\mathcal{F}$ be bounded by $2g$ distinct Jordan closed curves and take any curve connecting $p_i\in c_i,p_{i+g}\in c_{i+g}$ such that $p_{i+g}=\g_i(p_i)$ and $p_{i+g}\in c_{i+g}=\g_i(p_{i+g})$ for $1\le i\le g,$ that intersects $c_i,c_{i+g}$
orthogonally.
\end{proof}

It should be noted that a quasi-circle of $\G$ in general is not necessarily \emph{rectifiable}. For instance, if we take $\zeta$ to be some non-rectifiable generating curves then,
$\eta_\G$ will be non-rectifiable. Recall a curve is said to be rectifiable if and only if the $1$-dimensional Hausdorff measure of the curve is finite.
This is not the only obstruction to rectifiability, in fact we have the following result of Bowen:

\begin{thm}[\cite{Bowen}]\label{Bowen}
For a given Schottky group $\G,$ the Hausdorff dimension of limit set is $\mathfrak{D}_\G<1,$ if and only if there exists a rectifiable quasi-circle for $\G.$
\end{thm}
The proof of Theorem \ref{Bowen} relies on the fact that the Poincare series of $\G$ converges if and only if $\mathfrak{D}_\G<1$ \cite{Bowen}.
\begin{prop}\label{rec}
Let $\G$ be a Schottky group of $\mathfrak{D}_\G<1.$ Suppose a given generating curve $\zeta$ is a rectifiable curve. Then $\eta_\G$ is a rectifiable quasi-circle of $\G.$
\end{prop}
\begin{proof}
Let $\mu^1$ be the $1$-dimensional Hausdorff measure. Since $\mathfrak{D}_\G<1,$ we have $\mu^1(\Lambda_\G)=0.$ Let $\G=\{\g_k\}_{k=1}^\infty$, 
let $\g'_k$ denotes the derivative of $\g_k.$ Then we have,   
\begin{align*}
\mu^1(\eta_\G)&=\sum^\infty_{k=1}\mu^1(\g_k(\zeta)) \\
&\asymp\mu^1(\zeta)\sum^\infty_{k=1}|\g'_k(z)|,  \quad z\in \zeta.
\end{align*} 
Also since $\mathfrak{D}_\G<1$ if and only if Poincare series satisfies $\sum^\infty_{k=1}|\g'_k|<\infty.$ This implies that $\eta_\G$ is rectifiable if and only if 
$\zeta$ rectifiable.       
\end{proof}
Let $W\subset\mathbb{C}$ be a compact set. Denote the space of closed curves with \emph{bounded} length in $\mathbb{C}$ that intersect with $W$ by: 
\[\mathscr{R}(S^1,W)\subset\{h:S^1\to\mathbb{C}| h(S^1)\cap W\not=\emptyset,\quad\text{$h$ continuous rectifiable map}\}.\]
For $h_1,h_2\in\mathscr{R}(S^1,W)$ let $\ell(h_1),\ell(h_2)$ be it's respective arclength. The Fr\'{e}chet distance is defined as,
 \[d_F(h_1,h_2)=\inf\{\sup|h_1(\sigma_1)-h_2(\sigma_2)|;\sigma_1,\sigma_2\in\text{Homeo}(S^1)\}+|\ell(h_1)-\ell(h_2)|.\]
For a given compact $W\subset\mathbb{C},$ the space of closed curves with bounded length $\mathscr{R}(S^1,W)$ is a metric space with respect to $d_F.$ 
Two curves in $\phi,\psi\in\mathscr{R}(S^1,W)$ are same if there exists parametrization $\sigma$ such that $\psi(\sigma)=\phi$ and  $\ell(\phi)=\ell(\psi).$
The topology on $\mathscr{R}(S^1,W)$ is defined with respect to the metric $d_F$, see \cite{BR} p388. Let $\xi$ be a generating curve for $\eta_\G$. Fix a indexing of $\G$, set $\xi_i=\cup^i_1\g_j(\xi)$. Let $\sigma_i$ be a parametrization of $\xi_i$ such that 
$\sigma=\cup_i\sigma_i$ define a parametrization of $\eta_\G.$ Then $d_F(\eta^1_\G,\eta^2_\G)\le M(\inf_{\sigma^1,\sigma^2}|\xi(\sigma^1_k)-\zeta(\sigma_k^2)|+|\ell(\xi)-\ell(\zeta)|)$ for some $M, k.$ This implies continuity of $\eta_\G$ with respect to generating curve $\xi.$

\begin{prop}\label{curves-space}
For a given compact $W\subset\mathbb{C},$ the space $\mathscr{R}(S^1,W)$ is complete metric space with respect to $d_F.$
 \end{prop}
 \begin{proof}
Let $\{\phi_i\}\subset\mathscr{R}(S^1,W)$ be a Cauchy sequence. $\phi_i$ are rectifiable curves of bounded length and so there exists $\{\sigma_i\}$ Lipschitz parameterizations with bounded Lipschitz constants, such that
$\{\phi_i(\sigma_i(t))\}$ are uniformly Lipschitz. Then completeness follows from
the fact that all curves of $\mathscr{R}(S^1,W)$ are contained within some large compact subset of $\mathbb{C}.$
 \end{proof}
 \begin{prop}\label{H}
There exists a $\G\in\overline{\mathfrak{J}^{H_c}_g}$ in the closure of $\mathfrak{J}^{H_c}_g$, such that $\G$ is not a classical Schottky group. 
\end{prop}
\begin{proof}
Suppose false, then every element $\Gamma\in\overline{\mathfrak{J}^{H_c}_g}$  with Hausdorff dimension $\le H_c$ is a classical Schottky group.
Since the classical Schottky space $\mathfrak{J}_{g,o}$ is open in the Schottky space $\mathfrak{J}_g$,  we have a open neighborhood $U$ of 
$\overline{\mathfrak{J}^{H_c}_g}$ in $\mathfrak{J}_g$ such that $U\subset\mathfrak{J}^{H_c}_{g,o}.$ By definition of $H_c$, it is maximal. 
Hence there are non-classical Schottky groups of Hausdorff dimension arbitrarily close to $H_c$ and $>H_c$. Then there exists a sequence of 
non-classical Schottky groups $\{\Gamma_{n}\}$ of Hausdorff dimensions $\mathfrak{D}_{\Gamma_n}\searrow H_c.$ Let $\Gamma=\lim_n\Gamma_n.$ Since by assumption, all 
Schottky groups of Hausdorff dimension $H_c$ is classical,  we must have $\Gamma$ either it is \emph{not}
a Schottky group, or it is a classical Schottky group.  This implies that for large $n$ we must have $\Gamma_n\in U$, which is a contradiction to the sequence $\Gamma_n$ been
all non-classical Schottky groups, hence $H_c$ is not maximal.
\end{proof}

\begin{prop}\label{deformation}
There exists a non-classical Schottky group $\G$ of $\mathfrak{D}_\G=H_c,$ and sequence of $\{\G_n\}$ classical Schottky groups such that $\mathfrak{D}_{\G_n}<\mathfrak{D}_\G$ with $\sup \mathfrak{D}_{\G_n}=\mathfrak{D}_\G.$
\end{prop}
\begin{proof}
It follows from that the Hausdorff dimension map $\mathfrak{D}:\mathfrak{J}_g\to (0,2)$ is real analytic map, and by Proposition \ref{H}, there exists a non-classical $\G$ with
$\mathfrak{D}_\G=H_c.$ Theorem \ref{Hou} implies $\mathfrak{J}^{H_c}_{g}$ is open submanifold of $\mathfrak{J}_{g,o},$ hence there exists a sequence of classical Schottkys $\{\G_n\}\subset\mathfrak{J}^{H_c}_{g}$ with $\mathfrak{D}_{\G_n}\to\mathfrak{D}_\G.$
\end{proof}

\section{Schottky Space and rectifiable curves}
Take $\{\G_n\}$ and $\G$ as given by Proposition \ref{deformation}. In particular, $\{\G_n\}$ is a sequence of classical Schottky groups such that $\G_n\to\G.$ Denote $QC(\mathbb{C})$ space of quasiconformal maps on $\mathbb{C}$. It follows from quasiconformal deformation theory of Schottky space, for
$\G_c$ classical Schottky group of Hausdorff dimension $<H_c$ we can write ,
\[\mathfrak{J}_g=\{f\circ\gamma\circ f^{-1}\in\text{PSL}(2,\mathbb{C})|  f\in QC(\mathbb{C}), \gamma\in\Gamma_c\}/\mbox{PSL}(2,\mathbb{C}).\]

\noindent{\bf Remark:} 
From \emph{now on} throughout rest of the paper, we fix  $\G_c$ to be a classical Schottky group with Hausdorff dimension 
$\mathfrak{D}_{\G_c}<H_c.$ \\

\noindent{\bf Notations:} 
Set $\mathfrak{H}$ to be the collection of all Schottky groups of Hausdorff dimension $<H_c.$

Note that there exists a sequence of quasiconformal maps $f_n$ and $f$ of $\mathbb{C}$ such that, we can write
 $\G_n=f_n(\G)$ and $\G=f(\G_c).$ Here we write $f(\G):=\{f \circ g\circ f^{-1}| g\in \G\}$ for a given Kleinian group $\G$ and quasiconformal map $f.$\par
 
 Schottky space $\mathfrak{J}_g$ can also be considered as subspace of $\mathbb{C}^{3g-3}.$ This provides $\mathfrak{J}_g$ analytic structure as 
 $3g-3$-dimensional complex analytic manifold. 
 
 \begin{prop}\label{domain}
 Let $\{\G_n\}$ be a sequence of Schottky groups with $\G_n\to\G$ to a Schottky group $\G.$ Let $\mathcal{F}$ be a fundamental domain of $\G.$ There exists
 a sequence of fundamental domain $\{\mathcal{F}_n\}$ of $\G_n$ such that $\mathcal{F}_n\to\mathcal{F}.$ 
 \end{prop}
 
 \begin{proof}
 Let $\cup_{i=1}^{2g}\mathcal{C}_i=\partial\mathcal{F}$ be the Jordan curves which is the boundary of $\mathcal{F}.$ 
 Set $\mathcal{C}_n=f^{-1}_n(\cup_{i=1}^{2g}\mathcal{C}_i).$ Then $\mathcal{C}_n$ is the boundary of a fundamental domain of $f^{-1}_n(\G).$
 Hence we have a fundamental domain $\mathcal{F}_n$ of $\G_n$ defined by $\mathcal{C}_n$ with $\mathcal{F}_n\to\mathcal{F}.$
 \end{proof}

 \begin{lem}\label{sequence}
Suppose $f_n\to f$ and $\G=f(\G_c).$ Every quasi-circle with bounded length of $\G_n=f_n(\G_c)$ is in $\mathscr{R}(S^1,W)$ for some compact $W.$
 \end{lem}
 \begin{proof}
 Let $\eta_n$ be a quasi-circle of $\G_n.$ Note  $\Lambda_{\G_n}\subset\eta_n.$ Since limit set $\Lambda_{\G_n}\to\Lambda_{\G},$  and limit set $\Lambda(\G)$
 is compact, and $\eta_n$ is rectifiable, we can find some compact set $W\supset \cup_n\Lambda_{\G_n}\cup\Lambda_{\G}.$ 
 \end{proof}
 
Given $\{\G_n\}\subset\mathfrak{H}$ with $\G_n\to\G$, 
let $E_{\G_n}$ denote the collection of all bounded length quasi-circles of $\G_n.$ We define $\Psi_\G=\overline{\cup E_{\G_n}}^{d_F}$ to be the closure of the set of  bounded length quasi-circles of $\G_n$ in  $\mathscr{R}(S^1,W)$.
\begin{prop}\label{compact}
The subspace $\Psi_\G$ is compact.
\end{prop}
\begin{proof}
Curves in $\Psi_\G$ are bounded length and we have parametrization with bounded Lipschitz constants. It follows from Ascoli-Arzela theorem we have uniform convergence topology on $\Psi_\G$. 
Since curves in $\mathscr{R}(S^1,W)$ are all contained in some large compact set of $\mathbb{C},$ we have 
 $\Psi_\G$ is closed and bounded, hence compact. 
\end{proof}
We define $\mathscr{O}(\eta)$, open sets about $\eta\in\Psi_\G$ in relative topology given by $O(\eta)\cap\Psi_\G$ for some open set $O(\eta)\subset\mathscr{R}(S^1,W).$ \par

Let $\G$ be a Schottky group. Let $\mathcal{F}$ be a fundamental domain of $\G.$ For $\mathscr{O}(\eta_\G)$ of $\eta_\G\in\Psi_\G$, and suppose every element 
is quasi-circle, and let $\zeta_\xi$ denote a generating curve of  $\xi\in\mathscr{O}(\eta_\G)$ with respect to $\mathcal{F}.$ Then we have, $\mathscr{O}(\zeta)=\cup_{\xi\in\mathscr{O}(\eta_\G)}\zeta_\xi$ the collection of all
generating curves of the open set $\mathscr{O}(\eta_\G)$ gives a open set of generating curves of $\eta_\G.$ On set of collection of all generating curves of elements of $\Psi_\G$, we define the topology as $\xi_{\eta_n}\to\xi_\eta$, if and only if $\eta_n\to\eta$ for $\eta_n,\eta\in\Psi_\G.$
\par

 We will sometime denote by $\eta_\infty\in\partial\Psi_\G$ a curve which is the limit of rectifiable quasi-circles of $\{\G_n\}.$
 
 \begin{prop}\label{seq-curve}
Let $\{\G_n\}\subset\mathfrak{H}$ with $\G_n\to\G.$ 
Then every bounded length quasi-circle $\eta_\G$ of $\G$ is in $\Psi_\G.$ In addition, if $\eta_n$ are linear then $\eta_\G$ is  linear quasi-circles of $\G.$  \end{prop}
 
\begin{proof}
Let $\G_n=f^{-1}_n(\G).$ Note that since $\mathfrak{D}_{\G_n}<H_c$, we have $\mathfrak{D}_{\G}\le H_c<1.$
Define $\eta_n=f^{-1}_n(\eta_\G),$ for all $n.$ Then $\{\eta_n\}$ is a sequence of Jordan closed curves.
It follows from Proposition \ref{marking}, we have a generating curve $\zeta$ of $\eta_\G.$ So $\eta_\G=\Lambda_\G\cup\cup_{\g\in\G}\g(\zeta),$  and we have
$\eta_n=f^{-1}_n(\Lambda_\G)\cup f^{-1}_n(\cup_{\g\in\G}\g(\zeta)).$ Since $\Lambda_{\G_n}=f^{-1}_n(\Lambda_\G)$ and 
$f^{-1}_n(\cup_{\g\in\G}\g(\zeta))=\cup_{\g\in \G} f^{-1}_n\g f_n (f^{-1}_n(\zeta))$ so its $\cup_{\g_n\in\G_n}\g_n(\zeta_n),$ where $\zeta_n=f^{-1}_n(\zeta).$ Hence
$\zeta_n$ is a generating curve of $\eta_n$ which are quasi-circles of $\G_n$. Denote by $\zeta_\G$ a generating curve for $\eta_\G$. Since $\zeta$
is rectifiable curve, modify $\zeta_n$ if necessary,  we can assume $\{\zeta_n\}$ are rectifiable curves. \par
Let $z\in\zeta_n$, since $\mathfrak{D}_{\G_n}<H_c$, we have the $1$-dimension Hausdorff measure $\mu^1(\eta_n)$:
\[\mu^1(\eta_n)\asymp\mu^1(\zeta_n)\sum_{\g\in\G_n}|\g'(z)|<\infty,\]
where $\g'$ is the derivative of $\g.$
Hence $\{\eta_n\}$ are rectifiable quasi-circles. It follows that there exists $c>0$ such that $\mu^1(\eta_n)<c\mu^1(\eta_\G)$ for large $n$, hence 
$\{\eta_n\}$ are bounded quasi-circles of $\mathscr{R}(S^1,W)$, and by Proposition \ref{compact}, we have $\eta_n\to \eta_\G$ and  $\eta_\G\in\Psi_\G.$
Finally, if $\eta_n$ are linear then $\zeta_n$ are linear and since Mobius maps preserves linearity, we have $\eta_\G$ is linear. 
\end{proof}

 \begin{defn}[Good-sequences] \label{good}
 For a given sequence $\{\eta_n\}$ of quasi circles of $\{\G_n\}\subset\mathfrak{H}$ with $\G_n\to\G$, we say $\{\eta_n\}$ is a 
 \emph{good-sequence} of quasi-circles if, it is convergent sequence and $\eta_\infty$ is a quasi-circle of $\G.$ We also call $\{\eta_n\}$ (non)transverse good-sequence if all $\eta_n$ are also (non)transverse.
 \end{defn}
 
 \begin{lem}[Existence]\label{good}
 Let $\{\G_n\}\subset\mathfrak{H}$ be a sequence of Schottky groups with $\G_n\to \G.$ There exists a good-sequence 
 $\{\eta_n\}$ of quasi-circles of $\{\G_n\}.$ In addition, if $\{\eta_n\}$ is also (non)transverse then $\eta_\infty$ is (non)transverse.

 \end{lem}
 
 \begin{proof}
Let $f_n$ be quasi-conformal maps such that $\G_n=f_n(\G).$ Let $\eta_\G$ be a quasi-circle of $\G.$ Then $\eta_n=f_n(\eta_\G)$ is a good-sequence of quasi-circles. The (non)transverse property obviously is preserved. 

 \end{proof}
 
 \begin{cor}[Linear-invariant]\label{linear-invariant}
 Let $\{\eta_n\}$ be a good-sequence of linear quasi-circles of $\{\G_n\}\subset\mathfrak{H}.$  $\eta_\infty$ is a linear quasi-circle of $\G.$ 
 \end{cor}
 
 \begin{proof}
 Linearity is obviously preserved at $\eta_\infty.$
 \end{proof}
 
 \begin{cor}[Tranverse-invariant]\label{transverse}
  Let $\{\eta_n\}$ be a good-sequence of quasi-circles of $\{\G_n\}\subset\mathfrak{H}.$  Then $\eta_\infty$ is a (non)transverse linear quasi-circle of $\G$ if and only if $\{\eta_n\}$ is a (non)transverse. 
 \end{cor}
 \begin{proof}
 \end{proof}

\begin{lem}[Open]\label{open}
Let $\{\G_n\}\subset\mathfrak{H}$. Let $\G_n\to \G$ be a Schottky group. Let $\eta_\G$ be a quasi-circle of  $\G.$ Then
there exists a open neighboredood $\mathscr{O}(\eta_\G)$ of $\eta_\G$ in $\Psi_\G$ such that every elements of $\mathscr{O}(\eta_\G)$ is a quasi-circle of
$\G.$ 

\end{lem} 
\begin{proof}
Let $\mathcal{F}$ be a fundamental domain of $\G.$ Let $\xi$ be the generating curve of $\eta_\G$ with respect to $\mathcal{F}.$ 
Let $\mathcal{F}_n$ be the fundamental domain of $\G_n$ with $\mathcal{F}_n\to\mathcal{F}.$ By Proposition \ref{seq-curve}, we have $\eta_n\to\eta_\G.$ Let 
$\mathcal{U}(\xi_n)$ to be the open set about the generating curve $\xi_n$ of $\eta_n$, and set $\mathcal{O}(\xi_n)=\mathcal{U}(\xi_n)\cap\Psi_\G.$ Then $\mathcal{O}(\xi_n)$ is open
sets of generating curves in $\Psi_\G$ (Figure 2). Let $\mathcal{O}(\xi_n)=f_n(\mathcal{O}(\xi_c))$ for $\mathcal{O}(\xi_c)$ a open set of generating curves for 
$\eta_{\G_c}=f^{-1}_n(\eta_n).$
Since $\eta_\G$ is a quasi-circle, 
and $\mathcal{F}_n\to\mathcal{F}$, for large $n$ we can choose sufficiently small neighborhood $\mathcal{O}'(\xi_c)$ 
such that $f(\mathcal{O}'(\xi_c))$ is a open neighborhood of generating curves of $\eta_\G.$ 
Let $\xi'_n\in f_n(\mathcal{O}'(\xi_c))$ and $\eta_n'$ be generated by $\xi'_n.$ Assuming $\mathcal{O}'(\xi_c)$ is sufficiently small neighborhood, we
have length $\ell(\eta_n')<c\ell(\eta_\G)$ for some small $c\ge 1$ for large $n$ and $\eta_n'$ generated by all $\xi'_n\in f_n(\mathcal{O}'(\xi_c)).$ 
Since $\eta_n\to\eta_\G$, we have $\eta'_\infty$ is quasi-circle of $\G$ with bounded length. This defines a open set which all elements are quasi-circles in $\Psi_\G$. \end{proof}
\begin{figure}[ht!]
\centering
\includegraphics[width=35mm]{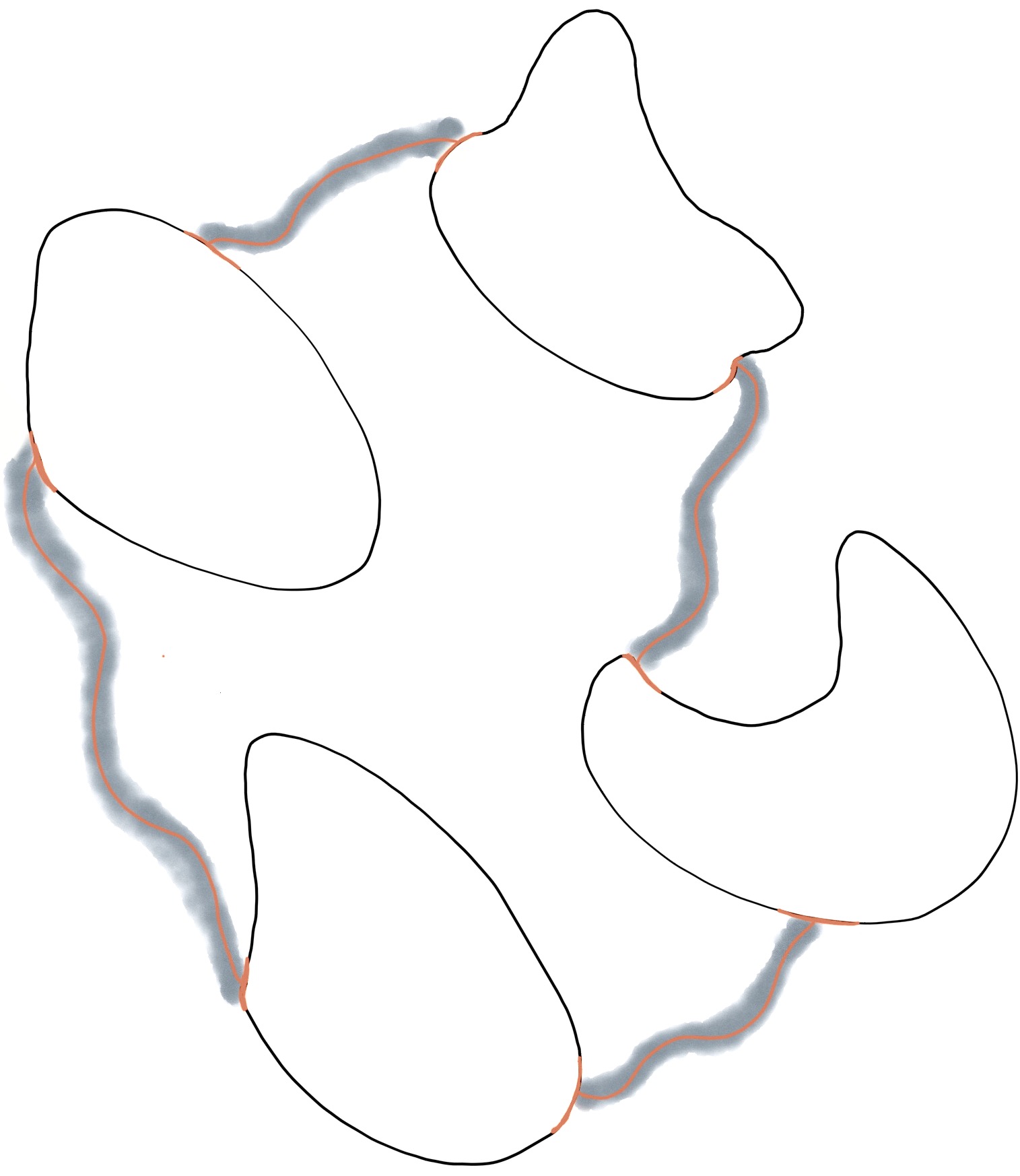}  
\caption{Open set of generating curves about $\zeta$ \label{overflow}}
\end{figure}
\begin{cor}\label{open-seq}
Let $\{\G_n\}\subset\mathfrak{H}$ with $\G_n\to\G$ a Schottky group $\G.$ Let $\{\eta_n\}$ be a good-sequence of quasi-circles of $\{\G_n\}.$ Then
there exists an open neighborhood $\mathscr{O}(\eta_\infty)$ of $\eta_\infty$ such that every element of $\mathscr{O}(\eta_\infty)$ is a quasi-circle of
$\G.$
\end{cor}
 \begin{proof}
Follows from Lemma \ref{good} and Lemma \ref{open}.
\end{proof}
Next we analyze the formations of singularities for a given sequence of classical Schottky groups converging to a Schottky group. These types of singularities has been
studied in \cite{Hou1,Hou}.
\begin{lem}[Singularity]\label{singular}
Let $\{\G_n\}\subset\mathfrak{H}$ with $\G_n\to\G$ a Schottky group. Assume that $\G$ is non-classical Schottky group.  Then there exists $\eta_\G$ such that,
every $\mathscr{O}(\eta_\G)$ contains a non-quasi-circle. 
\end{lem}
Here we say a closed curve is a \emph{non-quasi-circle}, if it's non-Jordan(contains a singularity point), or it's not $\G$-invariant.
\begin{proof}
For each $n$, let $\mathcal{F}_n$ be a classical fundamental domain of $\G_n.$ 
Given a sequence of classical fundamental domains $\{\mathcal{F}_n\}$ the convergence is consider as follows: $\{\partial\mathcal{F}_n\}$ is collection of 
$2g$ circles $\{c_{i,n}\}_{i=1,...,2g}$ in the Riemann sphere $\overline{\mathbb{C}}$, pass to a subsequence if necessary, then $\lim_n c_{i,n}$, is either a point or a circle. 
We say $\{\mathcal{F}_n\}$ convergents to $\mathcal{G}$, if $\mathcal{G}$ is a region that have boundary consists of 
$\lim_n c_{i,n}$ for each $i$, which necessarily is either a point or a circle. Note that $\mathcal{G}$ is not necessarily a fundamental domain, nor it's necessarily connected.\par
Let $\lim\mathcal{F}_{n}=\mathcal{G}.$ By assumption that $\G$ is not a classical Schottky group, we have $\mathcal{G}$ is not a classical fundamental domain of $\G.$ We have $\partial\mathcal{G}$ consists of circles or points. However these circles may not be disjoint. More precisely, we have the following possible degeneration of circles of $\partial\mathcal{F}_n$ which gives $\partial\mathcal{G}$ of at least one of following singularities types:
\begin{itemize}
\item{Tangency}:
Contains tangent circles. 
\item{Degeneration}:
Contains a circles degenerates into a point. 
\item{Collapsing}:
Contains two circles collapses into one circle. Here we have two concentric circles centered at origin and rest circles squeezed in between these two and
these two collapse into a single circle in $\mathcal{G}.$

\end{itemize}
Consider $\partial\mathcal{G}$ contains a {\bf tangency}:\par
Let $p$ be a tangency point.
Let $\eta_\G$ be a quasi-circle of $\G$ which pass through point $p.$ Assume the Lemma is false, then all sufficiently small open neighborhood $\mathscr{O}(\eta_\G)$ contains only quasi-circles of $\G.$  It follows from Proposition \ref{sequence}, we have a sequence with $\eta_n\to\eta_\G$.
Let $\xi_n$ be the generating curves of $\eta_n$ with respect to $\mathcal{F}_n.$ Define $\phi_n$ as follows. Note that $\xi_n\cap\partial\mathcal{F}_n$ consists of points or linear arcs. We define $\zeta_n$ be a generating curve with $\zeta_n\cap\partial\mathcal{F}_n$ to consists of linear arcs, with one of its end point to be the point that converges into a tangency in $\partial\mathcal{G}.$  In addition we also require the arcs on $\partial\mathcal{F}_n$ that with a end points which converges to tangency point be be connect by a arc between the other end points. It is clear that every 
$\mathscr{O}(\eta_n)\cap\mathcal{F}_n$ contains a generating curve of this property. Let $\phi_n$ be the quasi-circle generated by $\zeta_n,$ and set
$\phi=\lim\phi_n$, then $\phi$ contains a loop singularity at a point of tangency (Figure 3). Hence $\phi$ is non-quasi-circle of $\G.$  Since every $\mathscr{O}(\eta_\G)$
contains such a $\phi$ we must have the lemma to be true for this type of singularity.\par
\begin{figure}[ht!]
\centering
\includegraphics[width=35mm]{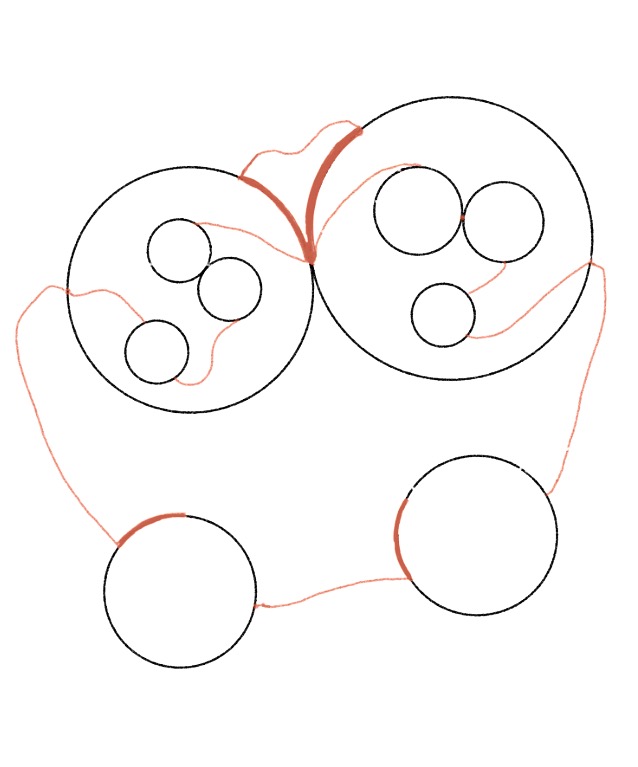}  
\caption{Tangency singularity \label{overflow}}
\end{figure}
\noindent Consider $\partial\mathcal{G}$ contains a {\bf degeneration}:\par
In this case, we must have any quasi-circle $\eta_n$ pass through a degeneration point. Note that we can't have all circles degenerates into a single point. There are two possibilities to have degenerate points. 
Case (A): two circles merge into a single degenerate point $p$. Case (B): A circle degenerates into a point on a circle. \par
Consider (A). In this case, any quasi-circles $\eta_\infty$ will have two possible properties, either there is a point $q$ on some circle of $\partial\mathcal{G}$ such that every quasi-circle
must pass through $q$, or two separate arcs of $\eta_\infty$ meet at $p$. The second possibility implies there exists a loop singularity at $p$, hence $\eta_\infty$ can not be a Jordan curve (Figure 4). Therefore we only have the first possibility for $\eta_\infty.$ But it follows from Proposition \ref{seq-curve}, all rectifiable quasi-circles of $\G$ is the limit
of some sequence of quasi-circles of $\G_n$, and hence all must pass through $q$. However $q$ is not a limit point, hence we can have some quasi-circle not passing through $q,$ a contradiction.\par
\begin{figure}[ht!]
\centering
\includegraphics[width=65mm]{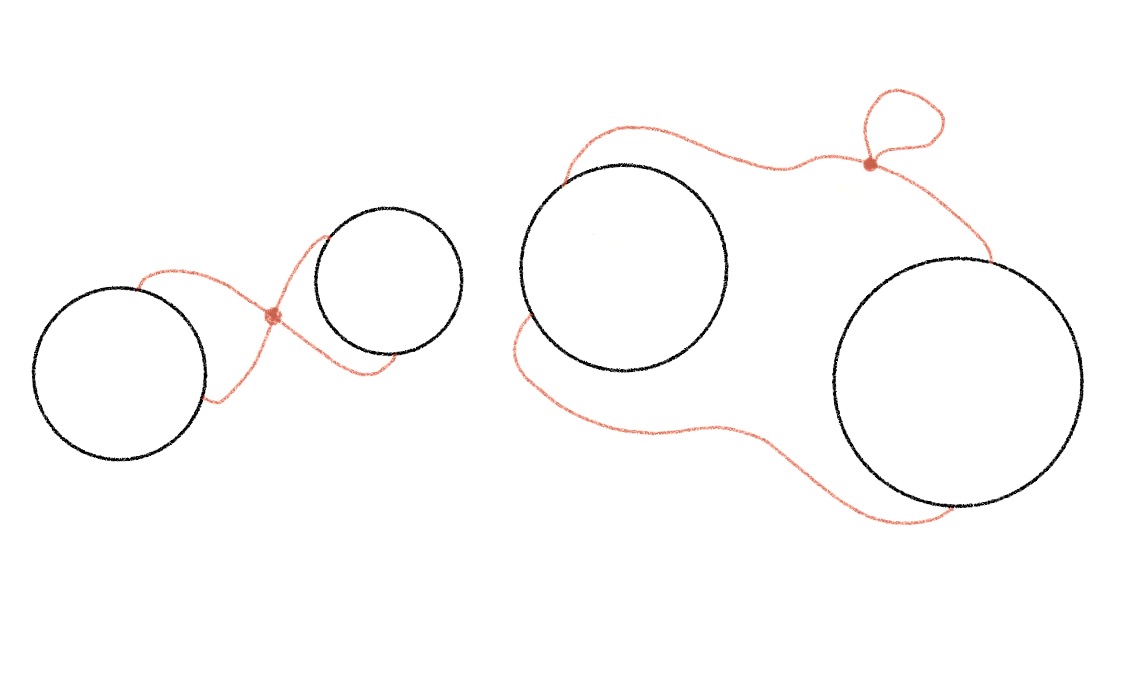}  
\caption{Degenerate singularities \label{overflow}}
\end{figure}

Now consider (B). Here we can assume that there exists at least two circles that do not degenerates into points. Otherwise, we will have the third type (collapsing) singularity, which we will consider next. Having some circles degenerates into a point on to a circle at $p$,  we have a sequence of  quasi-circles $\eta_n$ which passes through $p.$ This is given by the generating curves that have curves connecting the degenerating circle to point $p_n\to p$ converging to linear arc intersecting
orthogonally at boundary. But any neighboring quasi-circle of $\eta_n$  will converges to a $\eta_\infty$ with a loop singularity at $p,$ which is not a Jordan curve
(Figure 5). \par
\begin{figure}[ht!]
\centering
\includegraphics[width=30mm]{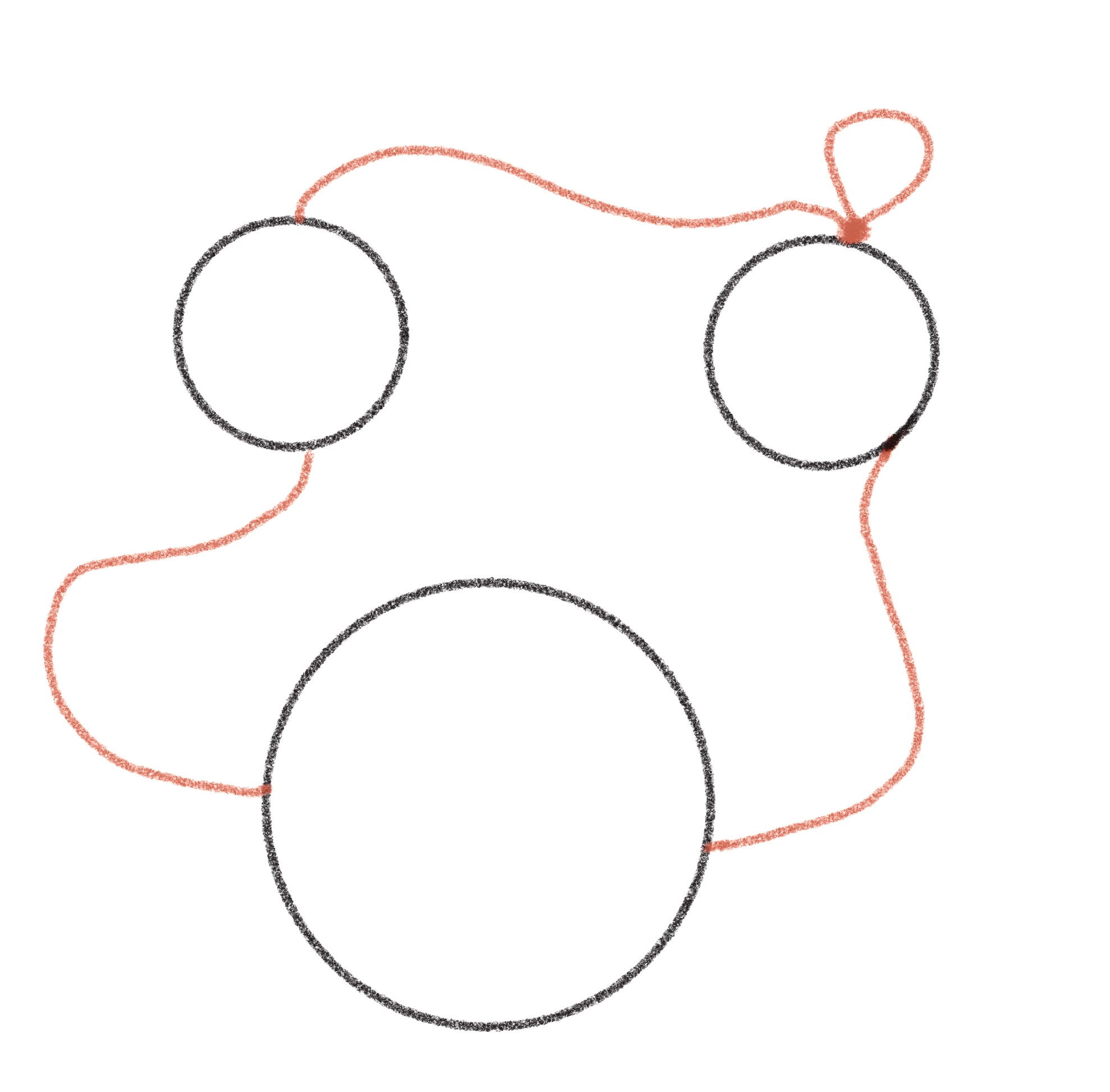}  
\caption{Degenerate singularities \label{overflow}}
\end{figure}

Finally, we note that if we have a degeneration point $p$ which is not a limit point, then there exists a neighboring quasi-circle of $\eta_\G$ that misses the point $p.$ Hence any $\mathscr{O}(\eta_\G)$ will contains some curve which is the limit of a sequence of quasi-circles that do not pass through the point. This gives a non-quasi-circle in $\mathscr{O}(\eta_\G).$\par
\noindent Consider $\partial\mathcal{G}$ contain {\bf collapsing}:\par
Let $\mathcal{C}$ denote the collapsed circle. First suppose that there exists $\g\in\G$ such that it has a fixed point not on 
$\mathcal{C}.$ Then there are infinitely many elements with fixed points not on $\mathcal{C}.$ Let $\g_n\in\G_n$ with $\g_n\to\g.$  
Since $\g_n(\mathcal{F}_n)\to\g(\mathcal{G})=\g(\mathcal{C}),$ we must have $\g(\mathcal{C})$ either identical or disjoint. Suppose 
that not all fixed points of elements of $\G$ is contained in $\mathcal{C}.$ Then we have infinitely many fixed points not in $\mathcal{C}.$ Take three points $a,b,c$ fixed points of elements of $\G$ with $a\in\mathcal{C}$ and 
$b,c\not\in\mathcal{C}.$ For sufficiently large $n$, we must have some $\beta_n\in\G_n$ such that $a,b,c$ is contained in three distinct disk of the complement of $\beta_n(\mathcal{F}_n).$ This follows from the fact that the orbit $\mathcal{F}_{\G_n}$ of $\mathcal{F}_n$ will have images with disks converging to fixed points, and
since they are distinct fixed points, we must have some disks that only contain one of the points only. Since $\beta_n(\mathcal{F}_n)\to\beta(\mathcal{C}),$ so converges to a circle. But $a,b,c$ are contained in distinct disks bounded by circles of $\partial\beta_n(\mathcal{F}_n)$ for all large $n$, which implies $a,b,c$ must lies on 
$\beta(\mathcal{C}).$ Hence $\beta(\mathcal{C})\cap\mathcal{C}\not=\emptyset$, but they are not identical circles, which is a contradiction. Hence all fixed points of elements of 
$\G$ are $\subset\mathcal{C},$ which implies $\Lambda_\G\subset\mathcal{C}.$ Since $\G$ is Schottky group, we must have $\G$ is Fuchsian group of second kind.
By \cite{Button}, $\G$ is classical Schottky group, a contradiction.

\end{proof}

\section{Proof of Theorem \ref{main}}
\begin{proof}(Theorem \ref{main})\vspace{0.2cm}\\
We proof by contradiction. Suppose that $H_c<1.$ \par
First note that by Selberg Lemma we can just assume Kleinian group to be torsion-free.\par
Now note that if a Kleinian group $\G$ of $\mathfrak{D}_\G<1$ then it must be free. To show this, assume otherwise. Since $\mathfrak{D}_\G<1$ and $\G$ is purely loxodromic, it is convex-cocompact of second-kind. There exists an imbedded surface 
$\mathscr{R}=\Omega_\G/\G$ in $\mathbb{H}^3/\G.$ If $\mathscr{R}$ is incompressible then, subgroup $\pi_1(\mathscr{R})\subset\G$ have $\mathfrak{D}_{\pi_1(\mathscr{R})}=1$ 
which is contradiction. If $\mathscr{R}$ is compressible then, we can cut along compression disks. We either end with incompressible surface as before or after finitely many steps of cutting we obtain topological ball, which implies $\mathbb{H}^3/\G$ is handle-body, hence free. \par

Let $\{\G_n\}\subset\mathfrak{H}$, and $\G_n\to\G$, with $\G$ a Schottky group. It follows from Lemma \ref{open}, that there exists an open set $\mathscr{O}(\eta_\G)$
such that every element is quasi-circle of $\G$. Now if $\G$ is non-classical Schottky group then by Lemma \ref{singular}, we must have every open set
contains some non-quasi-circle, in-particular we must have $\mathscr{O}(\eta_\G)$ contain a non-quasi-circle. But this gives a contradiction, hence we must
have $\G$ is a classical Schottky group.\par
Finally, sharpness comes from the fact that, there exists Kleinian groups which is not free of Hausdorff dimension equal to one. Hence we have our result.
\end{proof}
\text{}\\
E-mail: yonghou@princeton.edu
\pdfbookmark[1]{Reference}{Reference}
\bibliographystyle{plain}

\end{document}